\documentclass[12pt,reqno,a4paper]{amsart}
\usepackage{float}
\usepackage[bf,hypcap]{caption}
\usepackage{a4}
\usepackage{amsmath}
\usepackage{amsfonts}
\usepackage{amsthm}
\usepackage{graphicx}
\usepackage[colorlinks=true,citecolor=black,linkcolor=black,urlcolor=blue]{hyperref}

\makeatletter
\newtheorem{theorem}{Theorem}[section]

\newtheorem{conjecture}[theorem]{Conjecture}

\numberwithin{equation}{section}

\begin{document}

 \title[{\tiny Congruences modulo powers of 2 and 3 for overpartition $k$-tuples}]{Congruences modulo powers of 2 and 3 for overpartition $k$-tuples}
 \maketitle
\begin{center}
{G. Kavya Keerthana, S. Ananya and Ranganatha D.\\
 \vspace{.5cm}

Department of Mathematics, Central University of Karnataka\\
Kalaburagi-585376, India\\
E-mail: kavya.27598@gmail.com, ananyasahadev2525@gmail.com and ddranganatha@gmail.com}
 \end{center}
\begin{abstract} Let $\overline{p}_{k}(n)$ denote the number of overpartition $k$-tuples of $n$. In 2023, Saikia \cite{saikia} conjectured the following congruences: 
\begin{align*} 
\overline{p}_{q}(8n+2)& \equiv 0 \pmod{4},\quad
\overline{p}_{q}(8n+3)\equiv 0 \pmod{8},\quad 
\overline{p}_{q}(8n+4) \equiv 0 \pmod{2},\\
\overline{p}_{q}(8n+5)& \equiv 0 \pmod{8},\quad 
\overline{p}_{q}(8n+6) \equiv 0 \pmod{8},\quad
\overline{p}_{q}(8n+7)\equiv 0 \pmod{32},
\end{align*}
where $n\geq0$ and $q$ is prime. Recently, Sellers \cite{sellers2024elementary} showed that these congruences hold for all odd integers $q$ (not necessarily prime). In this paper,  we show that the above congruences hold for all positive integers $q$ (not necessarily odd). We also prove the following congruences on $\overline{OPT}_k(n)$, the number of overpartition $k$-tuples with odd parts of $n$: For all $i,j\geq 1$, $n\geq 0$,   $r$ not a multiple of 2, $k$ not a multiple of 2 or 3, and $\ell$ not a power of 2,  nor a multiple of 2 or 3, we have
\begin{align*}
\overline{OPT}_{2^i\cdot r}(8n+7)& \equiv 0 \pmod{2^{i+4}},\\
 \overline{OPT}_{3^i\cdot 2^j\cdot k}(3n+2)& \equiv 0 \pmod{3^{i+1}\cdot 2^{j+2}},\\
  \overline{OPT}_{3^i\cdot 2^j\cdot k}(3n+1)& \equiv 0 \pmod{3^{i}\cdot 2^{j+1}},\\
\overline{OPT}_{3^i\cdot \ell}(3n+2)& \equiv 0 \pmod{3^{i+1}\cdot 2},\\
 \overline{OPT}_{3^i\cdot \ell}(3n+1)& \equiv 0 \pmod{3^{i}\cdot 2},\end{align*}
 where the first congruence was posed as a conjecture by Sarma et al. \cite{saikiasarma} and the latter four were conjectured by Das et al. \cite{DSS}.
   
\vspace{2mm}
\vspace{2mm}
\end{abstract}

\section{Introduction}
In his book, MacMahon \cite{mac} has introduced a new combinatorial object, which has been extensively used to interpret several prominent identities like Ramanujan’s $_1\psi_1$ summation, Heine’s transformation, and Lebesgue’s identity.
Corteel and Lovejoy \cite{corteel} have laid the foundation for the premise of this object and named it overpartitions.
An overpartition of $ n$ is a non-increasing sequence of natural numbers whose sum is $n$ in which the first occurrence of a number may be overlined. Let $\overline{p}(n)$
 denote the number of overpartitions of $n$. For convenience, define $\overline{p}(0)=1$. Overpartitions are also known as standard MacMahon diagrams, joint partitions, superpartitions, dotted partitions, and jagged partitions (a term used particularly in theoretical physics).  
Several researchers discovered the divisibility properties of $\overline{p}(n)$. See \cite{dasappa2024further, liang2025school} and the references for recent works.

 Later, Lovejoy \cite{lovejoy} introduced the overpartition pairs, which also played a key role in proving similar identities as mentioned above. More recently, various arithmetic properties have been revealed for overpartition pairs. For more information and references in this direction, see \cite{bring, chen, kim}.
A natural generalization of overpartition pairs was given by Keister, Sellers and Vary \cite{keister} called the overpartition $k$-tuples, also identified as $k$-colored overpartitions. An overpartition $k$-tuple
of a positive integer $n$ is a $k$-tuple of overpartitions $(\pi_1,\pi_2,\ldots,\pi_k)$ such that the sum of the parts of $\pi_i$'s equal $n$. If $\overline{p}_{t}(n)$ is the number of overpartition $t$-tuples ($t$-colored overpartitions) of $n$, then its generating function is given by
\begin{align}\label{tt}
    \sum_{n=0}^{\infty}\overline{p}_{t}(n)q^n=\frac{f_2^t}{f_1^{2t}},
\end{align}
where $f_k=(q^k;q^k)_\infty=\displaystyle\prod_{n=1}^{\infty}(1-q^{kn})$, $|q|<1$.
Note that for $t=1$, we have $\overline{p}_{1}(n)=\overline{p}(n)$. \\

In 2010, Chen  \cite{chen2} gave a short proof of some congruences proved in \cite{keister} and obtained some new congruences for $\overline{p}_k(n)$ modulo prime $\ell$ and integer $2k$. In 2022, Nayaka and Naika \cite{nayaka2022congruences} proved several infinite families of congruences modulo 16 and 32 for $\overline{p}_{t}(n)$ with $t=5, 7, 11$ and $13$. For the same values of  $t$, by implementing Radu’s algorithm \cite{radu2015algorithmic}, Saikia \cite{saikia}  discovered several new congruences modulo powers of 2 for $\overline{p}_{t}(n)$ and some of his results improve  Nayaka and Naika \cite{nayaka2022congruences} results. In the same paper, Saikia \cite{saikia} proposed the following conjecture:
\begin{conjecture}\label{Con}
    For all $n\geq 0$ and primes $q$, we have
    \begin{align}
 \label{E1}   \overline{p}_{q}(8n+1)& \equiv 0 \pmod{2},\\
\label{E2} \overline{p}_{q}(8n+2)& \equiv 0 \pmod{4},\\
\label{E3}   \overline{p}_{q}(8n+3)& \equiv 0 \pmod{8},\\
\label{E4}  \overline{p}_{q}(8n+4)& \equiv 0 \pmod{2},\\
 \label{E5} \overline{p}_{q}(8n+5)& \equiv 0 \pmod{8},\\
\label{E6}  \overline{p}_{q}(8n+6)& \equiv 0 \pmod{8},\\ 
\label{E7} \overline{p}_{q}(8n+7)& \equiv 0 \pmod{32}.
    \end{align}
\end{conjecture}
Recently, Sellers \cite{sellers2024elementary} showed that the above conjecture holds for all odd $t$ (not necessarily prime).

In 2012, Lin \cite{Lin}  investigated the arithmetic properties of $\overline{po}_2(n)$, the number of overpartition pairs of $n$ into odd parts and its generating function is given by
\begin{align*}
    \sum_{n\geq 0}\overline{po}_2(n)q^n&= \frac{f_2^6}{f_1^4f_4^2}.
\end{align*}
He obtained several congruences modulo small powers of 2 and 3. Motivated by the Lin \cite{Lin} work,  Adiga and Ranganatha \cite{AdiRan} established several congruences for $\overline{po}_2(n)$ modulo higher powers of 2. Later Ahmed, Barman and Ray \cite{ABR} obtained congruences modulo 32, 64 and 128 for $\overline{po}_2(n)$ with the aid of Ramanujan’s theta function identities and some known identities of $t_k(n)$ for $k= 6, 7$, where $t_k(n)$ denotes the number of representations of $n$ as a sum of $k$ triangular numbers. \\

The generating function for $\overline{OPT}_k(n)$, the number of overpartition $k$-tuples of $n$ with odd parts is given by
\begin{align}\label{LL1}
    \sum_{n\geq 0}\overline{OPT}_k(n)q^n&= \frac{f_2^{3k}}{f_1^{2k}f_4^k}.
\end{align} 
Several mathematicians have studied the divisibility properties of $\overline{OPT}_k(n)$ for different values of $k$. Recently, Drema and Saikia \cite{drema} have established several infinite families of congruences modulo
small powers of 2 and 3 for  $\overline{OPT}_3(n)$
by employing $q$-series
and Ramanujan’s theta-function identities. Quite recently, Sarma, Saikia and Sellers \cite{saikiasarma} proved several congruences for $\overline{OPT}_3(n)$ and $\overline{OPT}_k(n)$. At the end of the paper, they conjectured seven congruences for $ \overline{OPT}_{2^i\cdot r}(n)$, where $r$ is odd and $i\geq 1$. Very recently, Das, Saikia and Sarma \cite{DSS} proved three of the seven congruences. The remaining four congruences are still not proved, which are given below:
    \begin{conjecture}\label{Con2}
    For all $i\geq 1$, $n\geq 0$ and odd $r>0$, we have
    \begin{align*}
 %   \overline{OPT}_{2^i\cdot r}(8n+1)& \equiv 0 \pmod{2^{i+1}},\\
\overline{OPT}_{2^i\cdot r}(8n+2)& \equiv 0 \pmod{2^{2i+1}},\\
%  \overline{OPT}_{2^i\cdot r}(8n+3)& \equiv 0 \pmod{2^{i+3}},\\
  \overline{OPT}_{2^i\cdot r}(8n+4)& \equiv 0 \pmod{2^{2i+4}},\\
 % \overline{OPT}_{2^i\cdot r}(8n+5)& \equiv 0 \pmod{2^{i+2}},\\
 \overline{OPT}_{2^i\cdot r}(8n+6)& \equiv 0 \pmod{2^{2i+3}},\\  \overline{OPT}_{2^i\cdot r}(8n+7)& \equiv 0 \pmod{2^{i+4}}.
    \end{align*}
\end{conjecture}
 Based on numerical evidence, Das, Saikia and Sarma \cite{DSS} posed the following conjectures:
 \begin{conjecture}\label{Con3}
    For all $i,j\geq 1$ and $k$ not a multiple of 2 or 3, we have
    \begin{align}
 \label{H1}   \overline{OPT}_{3^i\cdot 2^j\cdot k}(3n+2)& \equiv 0 \pmod{3^{i+1}\cdot 2^{j+2}}.
    \end{align}
\end{conjecture}
\begin{conjecture}\label{Con4}
    For all $i\geq 1$ and $j$ not a power of 2,  nor a multiple of 2 or divisible by 3, we have
    \begin{align}
 \label{H2}   \overline{OPT}_{3^i\cdot j}(3n+2)& \equiv 0 \pmod{3^{i+1}\cdot 2}.
    \end{align}
\end{conjecture}
\begin{conjecture}\label{Con5}
    For all $i,j\geq 1$ and $k$ not a multiple of 2 or 3, we have
    \begin{align}
 \label{H3}   \overline{OPT}_{3^i\cdot 2^j\cdot k}(3n+1)& \equiv 0 \pmod{3^{i}\cdot 2^{j+1}}.
    \end{align}
\end{conjecture}
\begin{conjecture}\label{Con6}
    For all $i\geq 1$ and $j$ not a power of 2,  nor a multiple of 2 or divisible by 3, we have
    \begin{align}
 \label{H4}   \overline{OPT}_{3^i\cdot j}(3n+1)& \equiv 0 \pmod{3^{i}\cdot 2}.
    \end{align}
\end{conjecture}

Motivated by the work of Sellers \cite{sellers2024elementary} on $\overline{p}_{t}(n)$, in Section 3 of this paper, we show that Conjecture \ref{Con} holds for even $t$ also (not necessarily odd), by providing a simple and unified proof that works for all positive integers $t$. The main results proved in Section 3 are stated below:
\begin{theorem}\label{T1}
For non-negative integers $t$, $k$, $\alpha$ and $n\geq 0$, we have
    \begin{align}
      \label{C1}  \overline{p}_{t}(n)& \equiv 0 \pmod{2},\qquad \quad \text{for all $n\geq 1$,}\\
       \label{C2}  \overline{p}_{t}\left(2^{2\alpha+2}n+2^{2\alpha+1}\right)& \equiv 0 \pmod{4},\\
       \label{C3}  \overline{p}_{t}\left(2^{2\alpha+2}n+3\cdot2^{2\alpha}\right)& \equiv 0 \pmod{4},\\
\label{C10}  \overline{p}_{t}\left(2^{2\alpha+3}n+5\cdot2^{2\alpha}\right)& \equiv 0 \pmod{4},\\
\label{11}\overline{p}_{t}(2^{2\alpha+3}n+2^{2\alpha})&\equiv \begin{cases}
       2 \pmod{4},\quad \text{when $t$ is odd and $n=k(k+1)/2$,}\\
         0\pmod{4},\quad \text{otherwise.}
   \end{cases}
   \end{align}
   \end{theorem}
\begin{theorem}\label{T2}
For non-negative integers $t$ and $n\geq 0$, we have
\begin{align}
\label{C5} \overline{p}_{t}(8n+5)& \equiv 0 \pmod{8},\\
 \label{C12} \overline{p}_{t}(8n+6)& \equiv 0 \pmod{8},\\
\label{C13} \overline{p}_{t}(16n+10)& \equiv 0 \pmod{8},\\
% \label{C14} \overline{p}_{t}(16n+14)& \equiv 0 \pmod{8},\\
 \label{C4} \overline{p}_{t}(4n+3)& \equiv 0 \pmod{8}.
 \end{align}
 \end{theorem}
\begin{theorem}\label{T3}
For non-negative integers $t$, $k$ and $n\geq 0$, we have
\begin{align}
 \label{C8}  \overline{p}_{t}(4n+3)& \equiv 0 \pmod{16}, \qquad\text{when $t=4k,4k+2,4k+3,$}\\    
\label{C6}  \overline{p}_{t}(8n+6)& \equiv 0 \pmod{16}, \qquad\text{when $t=4k,4k+2,4k+3,$}\\
\label{C16} \overline{p}_{t}(16n+14)& \equiv 0 \pmod{16}.
 \end{align}
 \end{theorem}
\begin{theorem}\label{T4} For non-negative integers $t$ and $n\geq0$, we have
    \begin{align}
         \label{C9} \overline{p}_{t}(8n+7)& \equiv 0 \pmod{32}.
    \end{align}
\end{theorem}
In Section 4, we prove the following theorems, which were conjectured by Sarma, Saikia and Sellers \cite{saikiasarma}, and Das, Saikia and Sarma \cite{DSS}.
\begin{theorem}\label{T11} For all $i\geq 1$, $n\geq 0$ and odd $r>0$, we have
\begin{align*}
\overline{OPT}_{2^i\cdot r}(8n+7)& \equiv 0 \pmod{2^{i+4}}.
\end{align*}
\end{theorem}
\begin{theorem}\label{tt1}
For all $i,j\geq 1$ and $k$ not a multiple of 2 or 3, the Conjecture \ref{Con3} and Conjecture \ref{Con5} hold.
\end{theorem}
\begin{theorem}\label{tt2}
 For all $i\geq 1$ and $j$ not a power of 2,  nor a multiple of 2 or divisible by 3, the Conjecture \ref{Con4} and Conjecture \ref{Con6} hold.
\end{theorem}
% \begin{theorem}\cite[Conjecture 10.1 and 10.3]{DSS}\label{Con3}
%     For all $i,j\geq 1$ and $k$ not a multiple of 2 or 3, we have 
%       \begin{align}
%  \label{H1}   \overline{OPT}_{3^i\cdot 2^j\cdot k}(3n+2)& \equiv 0 \pmod{3^{i+1}\cdot 2^{j+2}},\\
%   \label{H3}   \overline{OPT}_{3^i\cdot 2^j\cdot k}(3n+1)& \equiv 0 \pmod{3^{i}\cdot 2^{j+1}}.
%     \end{align}
% \end{theorem}
% \begin{theorem}\cite[Conjecture 10.2 and 10.4]{DSS}\label{Con4}
%     For all $i\geq 1$ and $j$ not a power of 2,  nor a multiple of 2 or divisible by 3, we have
%     \begin{align}
%  \label{H2}   \overline{OPT}_{3^i\cdot j}(3n+2)& \equiv 0 \pmod{3^{i+1}\cdot 2},\\
%   \label{H4}   \overline{OPT}_{3^i\cdot j}(3n+1)& \equiv 0 \pmod{3^{i}\cdot 2}.
%     \end{align}
% \end{theorem}
% The rest of the paper is organized as follows: Section \ref{Sec2} is devoted to collecting some
% $q$-series identities that are often used to prove our results. In Section 3, we prove the congruences modulo powers of 2 for overpartition $k$-tuples that generalises the results of Sellers \cite{sellers2024elementary}, subsequently the results of Saikia \cite{saikia} and  we also prove the congruence (\ref{F7}) from \cite{saikiasarma}. In Section 4, we prove the conjectures by Das et al. \cite{DSS} for overpartition $k$-tuples with odd parts for modulo powers of 2 and 3.
\section{Preliminary results}\label{Sec2}
Using the binomial theorem, it is easy to verify that for any prime $p$, 
\begin{align}
  \label{B1}  f_1^{p^k}\equiv f_p^{p^{k-1}}\pmod{p^k}.
\end{align}
  We recall some of the 2-, and 3-dissections from \cite{poq}, which we frequently use in our proofs.
    \begin{align}
     \label{D1}   f_1^2 &= \frac{f_2f_8^5}{f_4^2f_{16}^2}-2q\frac{f_2f_{16}^2}{f_8},\\
     \label{D4}\frac{1}{f_1^3}&  = a(q^3)+3qb(q^3)+9q^2c(q^3),\\
     \label{D3}\frac{f_1^2}{f_2}& = d(q^3)-2qg(q^3),\\
     \label{D2} f_1^3 & = h(q^3)-3qm(q^3),
    \end{align}
    where $a_n=a(q^n):=\displaystyle\sum_{j,k=-\infty}^\infty q^{n\cdot(j^2+jk+k^2)}$ is one of Borweins' cubic theta functions, $a(q)=a_1^2\frac{f_3^3}{f_1^{10}}$, $b(q)=a_1\frac{f_3^6}{f_1^{11}}$, $c(q)=\frac{f_3^9}{f_1^{12}}$, $ d(q)=\frac{f_3^2}{f_6}$, $g(q)=\frac{f_1f_6^2}{f_2f_3}$, $h(q)=a_1f_1$ and $m(q)=f_3^3$.
\section{Proofs of Theorems \ref{T2}, \ref{T3} and \ref{T4}}
\begin{proof}[Proof of Theorem \ref{T1}]
In view of (\ref{B1}), with $p=2$ and $k=1$, we have
    \begin{align*}
    \sum_{n=0}^\infty\overline{p}_{t}(n)q^n&=\frac{f_2^t}{f_1^{2t}}\equiv 1\pmod{2},
    \end{align*}
    which proves (\ref{C1}), and the congruences (\ref{E1}) and  (\ref{E4})  follows immediately. Now, again employing (\ref{B1}), with $p=2$ and $k=2$, we rewrite (\ref{tt}) as
    \begin{align}
        \label{M1}\sum_{n=0}^\infty\overline{p}_{t}(n)q^n=\frac{f_2^t}{f_1^{2t}}\cdot\frac{f_1^{2t}}{f_1^{2t}}= \frac{f_2^tf_1^{2t}}{f_1^{4t}}\equiv \frac{f_1^{2t}}{f_2^{t}}\pmod{4}.
    \end{align}
  
    On substituting (\ref{D1}) in (\ref{M1}), we get
\begin{align*}
 \sum_{n=0}^\infty\overline{p}_{t}(n)q^n\equiv \frac{1}{f_2^{t}}\left(\frac{f_2f_8^5}{f_4^2f_{16}^2}-2q\frac{f_2f_{16}^2}{f_8} \right)^t \pmod{4},
\end{align*}
 which is the same as
  \begin{align}
 \label{G4}  \sum_{n=0}^\infty\overline{p}_{t}(n)q^n\equiv \begin{cases}
        \frac{f_8^{5t}}{f_4^{2t}f_{16}^{2t}}\equiv 1 \pmod{4},\qquad\qquad\qquad\qquad\qquad\quad \text{when $t$ is even,}\\
         \frac{f_8^{5t}}{f_4^{2t}f_{16}^{2t}}+2q\frac{f_8^{5t-6}}{f_4^{2t-2}f_{16}^{2t-4}}\equiv \frac{f_8^t}{f_4^{2t}}+2qf_8^3\pmod{4},\quad \text{when $t$ is odd.}
   \end{cases}
\end{align}

 Extracting the terms involving $q^{4n+i}$, where $i\in\{2,3\}$, on both sides of (\ref{G4}), we get
\begin{align}
       \label{M3} \overline{p}_{t}(4n+i)& \equiv 0\pmod{4}, \quad \textrm{for ~~ any} ~~ t,
    \end{align}
 and the congruence (\ref{E2}) follows immediately. Extracting the terms involving  $q^{4n}$ on both sides of (\ref{G4}), we obtain
     \begin{align*}
   \sum_{n=0}^\infty\overline{p}_{t}(4n)q^n\equiv \begin{cases}
       1 \pmod{4},\qquad \text{when $t$ is even,}\\
         \frac{f_2^{t}}{f_1^{2t}}\pmod{4},\quad \text{when $t$ is odd.}
   \end{cases}
\end{align*}
From the above, we see that
\begin{align}
     \label{I1}\overline{p}_{t}(4n)& \equiv \overline{p}_{t}(n)\pmod{4}, \quad \textrm{for ~~ any} ~~ t.
\end{align}
From (\ref{M3}) and (\ref{I1}), congruences (\ref{C2}) and (\ref{C3}) are obtained. Now, extracting the terms involving  $q^{4n+1}$ on both sides of (\ref{G4}), we have
  \begin{align}
  \label{N1}\sum_{n=0}^\infty\overline{p}_{t}(4n+1)q^n\equiv \begin{cases}
         0 \pmod{4}, \quad\text{when $t$ is even,}\\
         2f_2^3\pmod{4},\quad \text{when $t$ is odd.}
   \end{cases}
\end{align}
From (\ref{I1}), (\ref{N1}) and the Jacobi's identity \cite[Equation 5, p. 257]{jacobi}, the congruences (\ref{C10}) and (\ref{11}) are obtained. 
\end{proof}
% \section{Congruences modulo 8 for $\overline{p}_{t}(n)$}
\
% \label{C7}   \overline{p}_{t}(4p^2n+4pj+p^2)& \equiv 0 \pmod{8},\qquad\quad \text{ $0\leq j \leq p-1$,}\\
      
%         \label{C13} \overline{p}_{t}(8n+6)& \equiv 0 \pmod{8},\\
%         \label{C8}  \overline{p}_{t}(4n+3)& \equiv 0
\begin{proof}[Proof of Theorem \ref{T2}]
Employing (\ref{B1}), with $p=2$ and $k=3$, and substituting  (\ref{D1}) in (\ref{tt}), we find that
    \begin{align*}
       \sum_{n=0}^\infty\overline{p}_{t}(n)q^n&=\frac{f_2^t}{f_1^{2t}}= \frac{f_2^tf_1^{6t}}{f_1^{8t}}\\&\equiv \frac{(f_1^{2})^{3t}}{f_2^{3t}}\equiv \left(\frac{f_8^5}{f_4^2f_{16}^2}-2q\frac{f_{16}^2}{f_8}\right)^{3t}\pmod{8},
         \end{align*}
 which is equivalent to 
  \begin{align*}
 \sum_{n=0}^\infty\overline{p}_{t}(n)q^n\equiv
 \begin{cases}
 \frac{f_8^{15t}}{f_4^{6t}f_{16}^{6t}}\pmod{8},  &\text{when $t=4k$},\\
      \frac{f_8^{15t}}{f_4^{6t}f_{16}^{6t}}+2q\frac{f_8^{15t-6}}{f_4^{6t-2}f_{16}^{6t-4}}+4q^2\frac{f_8^{15t-12}}{f_4^{6t-4}f_{16}^{6t-8}}\pmod{8}, &\text{when $t=4k+1$},\\
       \frac{f_8^{15t}}{f_4^{6t}f_{16}^{6t}}+4q\frac{f_8^{15t-6}}{f_4^{6t-2}f_{16}^{6t-4}}+4q^2\frac{f_8^{15t-12}}{f_4^{6t-4}f_{16}^{6t-8}}\pmod{8}, &\text{when $t=4k+2$},\\
       \frac{f_8^{15t}}{f_4^{6t}f_{16}^{6t}}+6q\frac{f_8^{15t-6}}{f_4^{6t-2}f_{16}^{6t-4}}\pmod{8}, &\text{when $t=4k+3$}.\\
   \end{cases}
\end{align*}
Applying (\ref{B1}) to the above, we find that
  \begin{align*}
\sum_{n=0}^\infty\overline{p}_{t}(n)q^n\equiv \begin{cases}
 1\pmod{8},  &\text{when $t=4k$},\\
       \frac{f_4^2f_8^3}{f_{16}^2}+2qf_8^3+4q^2f_{16}f_{32} \pmod{8}, &\text{when $t=4k+1$},\\
        \frac{f_8^2}{f_4^4}+4qf_8^3+4q^2f_{16}f_{32} \pmod{8}, &\text{when $t=4k+2$},\\
        \frac{f_4^6f_8}{f_{16}^2}-2qf_8^3\pmod{8}, &\text{when $t=4k+3$}.\\
   \end{cases}
\end{align*}
 Now, extracting the terms involving $q^{4n+j}$, $1\leq j\leq 3$ on both sides of the above, we get
  \begin{align*}
  \sum_{n=0}^\infty\overline{p}_{t}(4n+1)q^n &\equiv 
  \begin{cases}
0 \pmod{8},&\text{when $t=4k$,}\\
2if_2^3\pmod{8}, &\text{when $t=4k+i, ~~i\in\{1, 2, 3\}$,}
  \end{cases}\\
 \sum_{n=0}^\infty\overline{p}_{t}(4n+2)q^n &\equiv 
  \begin{cases}
0 \pmod{8},&\text{when $t=4k,~~4k+3$,}\\
4f_4^3\pmod{8}, &\text{when $t=4k+1, ~~4k+2$,}
  \end{cases}
  \end{align*}
  and (\ref{C4}). From the above two congruences,  we derive (\ref{C5}), (\ref{C12}) and (\ref{C13}). 
  \end{proof}

 \begin{proof}[Proof of Theorem \ref{T3}]
 Employing (\ref{B1}), with $p=2$ and $k=4$, and substituting  (\ref{D1}) in (\ref{tt}), we find that
  \begin{align*}
       \sum_{n=0}^\infty\overline{p}_{t}(n)q^n&=\frac{f_2^t}{f_1^{2t}}= \frac{f_2^tf_1^{14t}}{f_1^{16t}}\equiv \frac{(f_1^{2})^{7t}}{f_2^{7t}}\pmod{16}\\
       & \equiv \left(\frac{f_8^5}{f_4^2f_{16}^2}-2q\frac{f_{16}^2}{f_8}\right)^{7t}\pmod{16}.
         \end{align*}
  We see that the above congruence reduces to
  \begin{align}
 \sum_{n=0}^\infty\overline{p}_{t}(n)q^n&\equiv \frac{f_8^{35t}}{f_4^{14t}f_{16}^{14t}}+a_iq\frac{f_8^{35t-6}}{f_4^{14t-2}f_{16}^{14t-4}}+b_iq^2\frac{f_8^{35t-12}}{f_4^{14t-4}f_{16}^{14t-8}}\nonumber\\&~+c_iq^3\frac{f_8^{35t-18}}{f_4^{14t-6}f_{16}^{14t-12}}\pmod{16}, \label{G16}
\end{align}
where $t=8k+i$ and 
\begin{align*}
&(i, a_i, b_i, c_i)\in\left\{(0,0,0,0),(1,2,4,8),(2,4,12,0),(3,6,8,0),(4,8,8,0),(5,10,12,8),\right. \\&\left.(6,12,4,0),(7,14,0,0)\right\}.
  \end{align*}
 Now, extracting the terms involving  $q^{4n+i}$, $i\in\{2, 3\}$ on both sides of (\ref{G16}), we obtain
     \begin{align}
 \label{C17}\sum_{n=0}^\infty\overline{p}_{t}(4n+2)q^n&\equiv \begin{cases}
 0\pmod{16}, \qquad\qquad \text{when $t=8k,8k+7,$}\\
     4f_1^2f_2^5 \pmod{16},\qquad \text{when $t=8k+1,$}\\
      12f_2^2f_{4}^2 \pmod{16},\quad \text{when $t=8k+2$, }\\
      8f_{4}f_8 \pmod{16},\qquad \text{when $t=8k+3,8k+4,$}\\
      12f_1^2f_2^5 \pmod{16},\quad \text{when $t=8k+5,$}\\
      4f_2^2f_{4}^2 \pmod{16},\qquad \text{when $t=8k+6$, }
   \end{cases}
   \end{align}
   and
   \begin{align}
     \label{C18} \sum_{n=0}^\infty\overline{p}_{t}(4n+3)q^n&\equiv \begin{cases}
   8f_2f_{16}\pmod{16},\qquad \text{when $t=4k+1,$}\\
       0 \pmod{16},\qquad\qquad \text{otherwise}.
   \end{cases}
\end{align}
From (\ref{C18}), we obtain (\ref{C8}). Now, substituting (\ref{D1}) into (\ref{C17}) and then extracting the odd powers of $q$ on both sides, we find that
\begin{align*}
     \sum_{n=0}^\infty \overline{p}_{t}(8n+6)q^n& \equiv \begin{cases}
          0\pmod{16},\qquad\quad \text{when $t=4k,4k+2,4k+3$},\\
          8f_2f_8^2\pmod{16},\quad \text{when $t=4k+1$}.\\
      \end{cases} 
    \end{align*}
From the above, we obtain (\ref{C6}) and (\ref{C16}).
\end{proof}

% which proves the congruence (\ref{C8}).
%  Extracting the terms involving odd powers of $q$ from the above, we get
% \begin{align}
%       \overline{p}_{t}(8n+7)& \equiv 0\pmod{16}, \text{ for $n\geq 0$.}
%     \end{align}
%   which is congruence (\ref{C10}).
  
%   Now extracting the terms involving odd powers of $q$ from the above, we get
% \begin{align}
%       \overline{p}_{t}(8n+6)& \equiv \begin{cases}
%           0\pmod{16},\qquad \text{when $t=4k,4k+2,4k+3$}\\
%           8f_2f_8^2\pmod{16},\qquad \text{when $t=4k+1,$}\\
%       \end{cases} 
%     \end{align}
%      which is congruence (\ref{C6}). Now extracting the terms involving odd powers of $q$ from the above, we get
% \begin{align}
%       \overline{p}_{t}(16n+14)& \equiv 0\pmod{16},
%     \end{align}
%      which is congruence (\ref{C11}).
     
% \section{Congruences modulo 32 for $\overline{p}_{t}(n)$}
% \begin{theorem}For any  $n\geq0$, we have
%     \begin{align}
%          \label{C9} \overline{p}_{t}(8n+7)& \equiv 0 \pmod{32}, \qquad\text{when $t=4k+1,4k+2,4k+3$}.
%     \end{align}
% \end{theorem}
\begin{proof}[Proof of Theorem \ref{T4}] Employing (\ref{B1}), with $p=2$ and $k=5$, and substituting  (\ref{D1}) in (\ref{tt}), we find that
    \begin{align}
       \sum_{n=0}^\infty\overline{p}_{t}(n)q^n&=\frac{f_2^t}{f_1^{2t}}= \frac{f_2^tf_1^{30t}}{f_1^{32t}}\equiv \frac{(f_1^{2})^{15t}}{f_2^{15t}}\pmod{32}\nonumber\\
       & \equiv \left(\frac{f_8^5}{f_4^2f_{16}^2}-2q\frac{f_{16}^2}{f_8}\right)^{15t}\pmod{32}.
         \end{align}
  We can see that the above congruence reduces to
  \begin{align}
\label{G32} \sum_{n=0}^\infty\overline{p}_{t}(n)q^n&\equiv \frac{f_8^{75t}}{f_4^{30t}f_{16}^{30t}}+a_iq\frac{f_8^{75t-6}}{f_4^{30t-2}f_{16}^{30t-4}}+b_iq^2\frac{f_8^{75t-12}}{f_4^{30t-4}f_{16}^{30t-8}}+c_iq^3\frac{f_8^{75t-18}}{f_4^{30t-6}f_{16}^{30t-12}}\nonumber\\&~+d_iq^4\frac{f_8^{75t-24}}{f_4^{30t-8}f_{16}^{30t-16}} \pmod{32}, 
\end{align}
where $t=16k+i$ and
\begin{align*}
&(i, a_i, b_i, c_i, d_i)\in\left\{(0,0,0,0,0),(1,2,4,8,16),(2,4,12,0,16),(3,6,24,16,16),\right. \\
  &\left.(4,8,8,0,16),(5,10,28,24,0),(6,12,20,0,0),(7,14,16,0,0),(8,16,16,0,0),\right. \\
  &\left.(9,18,20,8,16),(10,20,28,0,16),(11,22,8,16,16),(12,24,24,0,16),\right.\\
   &\left.(13,26,12,24,0),(14,28,4,0,0),(15,30,0,0,0)\right\}.
\end{align*}
 Extracting the terms involving  $q^{4n+3}$  on both sides of (\ref{G32}), we obtain
     \begin{align}
   \label{M8}\sum_{n=0}^\infty\overline{p}_{t}(4n+3)q^n\equiv \begin{cases}
   m_jf_2f_{8}^2\pmod{32},\quad\quad \text{when $t=16k+j$,}\\
       0 \pmod{32},\qquad\quad\qquad \text{when $t=16k+i$,}\\
       16f_2f_{16}\pmod{32}\quad\qquad \text{when $t=16k+3,16k+11$},
   \end{cases}
\end{align}
where $(j,m_j)\in\{(1,8),(5,24),(9,8),(13,24)\}$ and $i\in\{0,2,4,6,7,8,10,12,14,15\}$.
 Extracting the terms involving odd powers of $q$ on both sides of (\ref{M8}), we get (\ref{C9}), and the congruence (\ref{E7}) follows directly.
\end{proof}

 \section{Proofs of Theorems \ref{T11},  \ref{tt1} and \ref{tt2}}
\begin{proof}[Proof of Theorem \ref{T11}]
Employing (\ref{B1}) with $p=2$ in (\ref{LL1}) with $k=i+4$  and then substituting  (\ref{D1}), we find that
  \begin{align*}
        \sum_{n\geq 0}\overline{OPT}_{2^i\cdot r}(n)q^n&= \frac{f_2^{3\cdot 2^i\cdot r}}{f_1^{2^{i+1}\cdot r}f_4^{2^i\cdot r}}\cdot \frac{f_1^{2^{i+1}\cdot 7r}}{f_1^{2^{i+1}\cdot 7r}}= \frac{f_1^{2^{i+1}\cdot 7r}f_2^{2^i\cdot 3r}}{f_1^{2^{i+4}\cdot r}f_4^{2^i\cdot r}}\\
        & \equiv  \frac{f_1^{2^{i+1}\cdot 7r}}{f_2^{2^{i}\cdot 5r}f_4^{2^i\cdot r}}\equiv \frac{(f_1^2)^{2^i\cdot 7r}}{f_2^{2^{i}\cdot 5r}f_4^{2^i\cdot r}}
       \equiv \frac{f_2^{2^{i+1}\cdot r}}{f_4^{2^i\cdot r}}\left( \frac{f_8^5}{f_4^2f_{16}^2}-2q\frac{f_{16}^2}{f_8}\right)^{2^{i}\cdot 7r}\pmod{2^{i+4}}\\
    & \equiv \frac{f_2^{2^{i+1}\cdot r}f_8^{2^i\cdot 35r}}{f_4^{2^{i}\cdot 15r}f_{16}^{2^{i+1}\cdot 7r}}-(2^{i+1}\cdot 7r)\frac{qf_2^{2^{i+1}\cdot r}f_8^{2^i\cdot 35r-6}}{f_4^{2^{i}\cdot 15r-2}f_{16}^{2^{i+1}\cdot 7r-4}}+(2^{i+1}\cdot 7r)(2^i\cdot 7r-1)\\ 
     &~\times\frac{q^2f_2^{2^{i+1}\cdot r}f_8^{2^i\cdot 35r-12}}{f_4^{2^{i}\cdot 15r-4}f_{16}^{2^{i+1}\cdot 7r-8}}
             -\frac{(2^{i+3}\cdot 7r)(2^i\cdot 7r-1)
        (2^{i-1}\cdot 7r-1)}{3}\\&~\times\frac{q^3f_2^{2^{i+1}\cdot r}f_8^{2^i\cdot 35r-18}}{f_4^{2^{i}\cdot 15r-6}f_{16}^{2^{i+1}\cdot 7r-12}}+\frac{(2^{i+2}\cdot 7r)(2^i\cdot 7r-1)(2^{i-1}\cdot 7r-1)(2^i\cdot 7r-3)}{3}\\ &~\times\frac{q^4f_2^{2^{i+1}\cdot r}f_8^{2^i\cdot 35r-24}}{f_4^{2^{i}\cdot 15r-8}f_{16}^{2^{i+1}\cdot 7r-16}}\pmod{2^{i+4}}\\
         & ~\equiv \frac{f_2^{2^{i+1}\cdot r}f_8^{2^i\cdot 35r}}{f_4^{2^{i}\cdot 15r}f_{16}^{2^{i+1}\cdot 7r}}-(2^{i+1}\cdot 7r)q\frac{f_2^{2^{i+1}\cdot r}f_8^{2^i\cdot 3r+2}}{f_4^{2^i\cdot 7r-2}}+(2^{i+1}\cdot 7r)(2^i\cdot 7r-1)\\ &~\times q^2\frac{f_2^{2^{i+1}\cdot r}f_8^{2^i\cdot 3r+4}}{f_4^{2^i\cdot 7r-4}}
          -\frac{(2^{i+3}\cdot 7r)(2^i\cdot 7r-1)(2^{i-1}\cdot 7r-1)}{3}q^3f_8^{9}
        \\ &~+\frac{(2^{i+2}\cdot 7r)(2^i\cdot 7r-1)(2^{i-1}\cdot 7r-1)(2^i\cdot 7r-3)}{3}q^4f_8^{12}\pmod{2^{i+4}}.
    \end{align*}
    Extracting the terms involving $q^{2n+1}$ on both sides of the above, we see that 
    \begin{align}
       \label{R1} \sum_{n\geq 0}\overline{OPT}_{2\cdot r}(2n+1)q^n& \equiv-4k\frac{f_1^{4r}f_4^{2r+2}}{f_2^{6r-2}}-16k'qf_4^9\pmod{32},\quad i=1, 
       \end{align}
       and 
       \begin{align}
       \label{T5} \sum_{n\geq 0}\overline{OPT}_{2^i\cdot r}(2n+1)q^n& \equiv-2^{i+1}\cdot 7rf_2^2f_4^2-2^{i+3}mqf_4^9\pmod{2^{i+4}},\quad i\geq2,
       \end{align}
  where $k=7r, k'=\frac{7r(14r-1)(7r-1)}{3}$ and $m=\frac{(7r)(2^i\cdot 7r-1)(2^{i-1}\cdot 7r-1)}{3}$. 
    Applying (\ref{B1}) to (\ref{R1}) and then substituting (\ref{D1}), we find that
     \begin{align*}
        \sum_{n\geq 0}\overline{OPT}_{2\cdot r}(2n+1)q^n& \equiv -4k\frac{f_1^{4}f_4^{4}}{f_2^{4}}-16k'qf_4^9\pmod{32}\\
           & \equiv -4k\frac{f_4^{4}}{f_2^{2}}\left(\frac{f_8^{10}}{f_4^4f_{16}^4}-4q\frac{f_{8}^4}{f_4^2}+4q^2\frac{f_{16}^4}{f_8^2}\right)-16k'qf_4^9\pmod{32}.
    \end{align*}
     Extracting the terms that involve $q^{2n+1}$ from both sides of the above and then applying (\ref{B1}), we obtain
    \begin{align}\label{T6}
        \sum_{n\geq 0}\overline{OPT}_{2\cdot r}(4n+3)q^n& \equiv 16kf_2f_4^2-16k'f_2^9\pmod{32}.
    \end{align}
    From (\ref{T5}) and (\ref{T6}), we see that Theorem \ref{T11} is true for all $i\geq1$. 
\end{proof}
 
 \begin{proof}[Proof of Theorem \ref{tt1}]
    From (\ref{LL1}), (\ref{B1}), (\ref{D4}) and (\ref{D2}), we find that
    \begin{align*}
        \sum_{n\geq 0}\overline{OPT}_{3^i\cdot 2^j\cdot k}(n)q^n&= \frac{f_2^{3^{i+1}\cdot 2^j\cdot k}}{f_1^{3^i\cdot 2^{j+1}\cdot k}f_4^{3^i\cdot 2^j\cdot k}}\times \frac{f_1^{3^i\cdot 2^{j}\cdot k}}{f_1^{3^i\cdot 2^{j}\cdot k}}\\
        &\equiv \frac{f_1^{3^i\cdot 2^{j}\cdot k}f_6^{3^{i}\cdot 2^j\cdot k}}{f_3^{3^i\cdot 2^{j}\cdot k}f_4^{3^i\cdot 2^j\cdot k}}\equiv \left(\frac{f_6}{f_3}\times \frac{f_1}{f_4}\right)^{3^{i}\cdot 2^j\cdot k}\pmod{3^{i+1}}\\
        & \equiv \left(\frac{f_6}{f_3}\right)^{3^{i}\cdot 2^j\cdot k} \left(\frac{f_1^3}{f_4^3}\right)^{3^{i-1}\cdot 2^j\cdot k}\pmod{3^{i+1}}\\
         & \equiv \left(\frac{f_6}{f_3}\right)^{3^{i}\cdot 2^j\cdot k}  \bigg[ (h(q^3)-3qm(q^3))  \\&~\times \left( a(q^{12})+3q^4b(q^{12})+9q^8c(q^{12})\right)\bigg ]^{3^{i-1}\cdot 2^j\cdot k}\pmod{3^{i+1}}\\
        & \equiv \left(\frac{f_6}{f_3}\right)^{3^{i}\cdot 2^j\cdot k} \left( t-3q\ell(q)\right)^{3^{i-1}\cdot 2^j\cdot k}\pmod{3^{i+1}},
    \end{align*}
    where $t=h(q^3)a(q^{12})$ and $\ell(q)=a(q^{12})m(q^3)-q^3b(q^{12})h(q^3)+3q^4b(q^{12})m(q^3)-3q^7c(q^{12})h(q^3)+9q^8c(q^{12})m(q^3)$. From the above, we observe that
  \begin{align}
        \sum_{n\geq 0}\overline{OPT}_{3^i\cdot 2^j\cdot k}(n)q^n&\equiv \left(\frac{f_6}{f_3}\right)^{3^{i}\cdot 2^j\cdot k}\left(\sum_{r=0}^{3^{i-1}\cdot 2^j\cdot k}\binom{3^{i-1}\cdot 2^j\cdot k}{r}t^{3^{i-1}\cdot 2^j\cdot k-r}(-3q\ell(q))^r\right)\nonumber\\
         & \equiv \left(\frac{f_6}{f_3}\right)^{3^{i}\cdot 2^j\cdot k}\left( t^{3^{i-1}\cdot 2^j\cdot k}+(3^{i-1}\cdot 2^j\cdot k)t^{3^{i-1}\cdot 2^j\cdot k-1}(-3q\ell(q))\right)\nonumber\\
         & \equiv \left(\frac{f_6}{f_3}\right)^{3^{i}\cdot 2^j\cdot k} \biggl(t^{3^{i-1}\cdot 2^j\cdot k}-(3^{i}\cdot 2^j\cdot k)t^{3^{i-1}\cdot 2^j\cdot k-1}\nonumber\\ &~\times\left(qa(q^{12})m(q^3)-q^4b(q^{12})h(q^3)\right)\biggr)\pmod{3^{i+1}}.\label{R2}
        \end{align}
    Extracting the terms involving $q^{3n+2}$ on both sides of the above, we obtain
    \begin{align}
       \label{R14} \overline{OPT}_{3^i\cdot 2^j\cdot k}(3n+2)\equiv 0\pmod{3^{i+1}}.
    \end{align}
    Next, we have
      \begin{align*}
        \sum_{n\geq 0}\overline{OPT}_{3^i\cdot 2^j\cdot k}(n)q^n&= \frac{f_2^{3^{i+1}\cdot 2^j\cdot k}}{f_1^{3^i\cdot 2^{j+1}\cdot k}f_4^{3^i\cdot 2^j\cdot k}}\times \frac{f_1^{3^i\cdot 2^{j+1}\cdot k}}{f_1^{3^i\cdot 2^{j+1}\cdot k}}\\
        &\equiv \frac{f_1^{3^i\cdot 2^{j+1}\cdot k}f_2^{3^{i}\cdot 2^j\cdot k}}{f_4^{3^i\cdot 2^j\cdot k}}\equiv \left(\frac{f_2^2}{f_4}\times \frac{f_1^2}{f_2}\right)^{3^{i}\cdot 2^j\cdot k}\pmod{2^{j+2}}\\
        & \equiv \left[(d(q^6)-2q^2g(q^6))\times (d(q^3)-2qg(q^3))\right]^{3^{i}\cdot 2^j\cdot k} \pmod{2^{j+2}}\\
           & \equiv  \left( t'(q^3)-2q\ell'(q)\right)^{3^{i}\cdot 2^j\cdot k}\pmod{2^{j+2}},
    \end{align*}
    where $t'(q)=d(q)d(q^{2})$ and $\ell'(q)=g(q^3)d(q^{6})+qd(q^3)g(q^{6})-2q^2g(q^3)g(q^{6})$. From the above, we have
  \begin{align}
        \sum_{n\geq 0}\overline{OPT}_{3^i\cdot 2^j\cdot k}(n)q^n&\equiv \sum_{r=0}^{3^{i}\cdot 2^j\cdot k}\binom{3^{i}\cdot 2^j\cdot k}{r}t'(q^3)^{3^{i}\cdot 2^j\cdot k-r}(-2q\ell'(q))^r\pmod{2^{j+2}}\nonumber\\
         & \equiv  t'(q^3)^{3^{i}\cdot 2^j\cdot k}+(3^{i}\cdot 2^j\cdot k)t'(q^3)^{3^{i}\cdot 2^j\cdot k-1}(-2q\ell'(q))\nonumber\\
         &~+\frac{(3^{i}\cdot 2^j\cdot k)(3^{i}\cdot 2^j\cdot k-1)}{2} t'(q^3)^{3^{i}\cdot 2^j\cdot k-2}(-2q\ell'(q))^2\pmod{2^{j+2}}\nonumber\\
         & \equiv t'(q^3)^{3^{i}\cdot 2^j\cdot k}-(3^{i}\cdot 2^{j+1}\cdot k)(t'(q^3)^{3^{i}\cdot 2^j\cdot k-1})\bigl(qg(q^{3})d(q^6)\nonumber\\&~+q^2d(q^{3})g(q^6)\bigr)-(3^{i}\cdot 2^{j+1}\cdot k)(t'(q^3)^{3^{i}\cdot 2^j\cdot k-2})\nonumber\\ & ~\times \left(q^2g^2(q^{3})d^2(q^6)+q^4d^2(q^{3})g^2(q^6)\right)\pmod{2^{j+2}}.\label{R3}
        \end{align}
    Extracting the terms involving $q^{3n+2}$ from both sides of the above, we obtain
    \begin{align}
        \sum_{n\geq 0} \overline{OPT}_{3^i\cdot 2^j\cdot k}(3n+2)q^n&\equiv -(3^{i}\cdot 2^{j+1}\cdot k)(t'(q)^{3^{i}\cdot 2^j\cdot k-2})\left(\frac{f_3^4f_2f_{12}}{f_4f_6}+\frac{f_1^2f_6^8}{f_2^2f_3^2f_{12}^2}\right)\nonumber\\
        &\equiv -(3^{i}\cdot 2^{j+1}\cdot k)(t'(q)^{3^{i}\cdot 2^j\cdot k-2})\left(2\frac{f_3^6}{f_1^2}\right)\nonumber\\& \equiv 0\pmod{2^{j+2}}.\label{R13}
    \end{align}
    Since $\gcd(3^{i+1},2^{j+2})=1$, from (\ref{R14}) and (\ref{R13}), we get
     \begin{align*}
       \overline{OPT}_{3^i\cdot 2^j\cdot k}(3n+2)\equiv 0\pmod{3^{i+1}\cdot 2^{j+2}}.
    \end{align*}
    This completes the proof of the congruence (\ref{H1}). Now, we prove the congruence (\ref{H3}). From (\ref{R2}), we find that
    \begin{align*}
        \sum_{n\geq 0}\overline{OPT}_{3^i\cdot 2^j\cdot k}(n)q^n& \equiv \left(\frac{f_6}{f_3}\right)^{3^{i}\cdot 2^j\cdot k} t^{3^{i-1}\cdot 2^j\cdot k}\pmod{3^{i}}.
        \end{align*}
    Extracting the terms involving $q^{3n+1}$ from both sides of the above, we obtain
    \begin{align}
       \label{R10} \overline{OPT}_{3^i\cdot 2^j\cdot k}(3n+1)\equiv 0\pmod{3^{i}}.
    \end{align}
       To prove the second part of the congruence (\ref{H3}), we can easily verify that
     \begin{align}
        \sum_{n\geq 0}\overline{OPT}_{3^i\cdot 2^j\cdot k}(n)q^n&= \frac{f_2^{3^{i+1}\cdot 2^j\cdot k}}{f_1^{3^i\cdot 2^{j+1}\cdot k}f_4^{3^i\cdot 2^j\cdot k}}\nonumber\\
        & \equiv \frac{f_2^{3^i\cdot 2^{j+1}\cdot k}}{f_4^{3^i\cdot 2^j\cdot k}} \equiv 1\pmod{2^{j+1}}.\label{R9}
    \end{align}
    Extracting the terms involving $q^{3n+1}$ and $q^{3n+2}$ from both sides of (\ref{R9}), we obtain
    \begin{align}
       \label{R11} \overline{OPT}_{3^i\cdot 2^j\cdot k}(3n+1)&\equiv 0\pmod{2^{j+1}}
       \end{align}
       and
       \begin{align}
       \label{R12} \overline{OPT}_{3^i\cdot 2^j\cdot k}(3n+2)&\equiv 0\pmod{2^{j+1}}.
    \end{align}
    Since, $\gcd(3^{i},2^{j+1})=1$, from (\ref{R9}) and (\ref{R11}) we get (\ref{H3}). This completes the proof of Theorem \ref{tt1}.
 \end{proof}
  \begin{proof}[Proof of Theorem \ref{tt2}]
 The proofs  of the following congruences are analogues to (\ref{R14}) and (\ref{R10}), respectively and are omitted: 
    \begin{align}
       \label{R4} \overline{OPT}_{3^i\cdot j}(3n+2)\equiv 0\pmod{3^{i+1}}
    \end{align}
    and 
    \begin{align}
     \label{R8}   \overline{OPT}_{3^i\cdot j}(3n+1)\equiv 0\pmod{3^{i}}.
    \end{align}
Second part of the congruence (\ref{H2}) follows directly, since
     \begin{align}
        \sum_{n\geq 0}\overline{OPT}_{3^i\cdot j}(n)q^n&= \frac{f_2^{3^{i+1}\cdot j}}{f_1^{3^i\cdot 2j}f_4^{3^i\cdot j}}\nonumber\\
        & \equiv \frac{f_2^{3^i\cdot 2\cdot j}}{f_4^{3^i\cdot j}} \equiv 1\pmod{2}.\label{R5}
    \end{align}
    Extracting the terms involving $q^{3n+1}$ and $q^{3n+2}$ from both sides of (\ref{R5}), we obtain
    \begin{align}
       \label{R6} \overline{OPT}_{3^i\cdot j}(3n+1)&\equiv 0\pmod{2},\\
       \label{R7} \overline{OPT}_{3^i\cdot j}(3n+2)&\equiv 0\pmod{2}.
    \end{align}
    Since $\gcd(3^{i+1},2)=1$, from (\ref{R4}) and (\ref{R7}), we arrive at the congruence (\ref{H2}). Since $\gcd(3^{i},2)=1$, from (\ref{R8}) and (\ref{R6}) we obtain (\ref{H4}). This completes the proof of Theorem \ref{tt2}.
 \end{proof}
 \bibliographystyle{siam}
 \bibliography{myref}

\end{document}